\documentclass[reqno]{amsart}

\usepackage{amsmath,amssymb,amscd,amsthm,accents}
\usepackage{graphicx} \DeclareGraphicsExtensions{.eps}
\usepackage{mathrsfs} \usepackage[mathcal]{eucal}
\usepackage{enumitem}

\def\sideremark#1{\ifvmode\leavevmode\fi\vadjust{\vbox to0pt{\vss
      \hbox to 0pt{\hskip\hsize\hskip1em
        \vbox{\hsize3cm\tiny\raggedright\pretolerance10000%
          \noindent #1\hfill}\hss}\vbox to8pt{\vfil}\vss}}}%

\newtheorem{Thm}{Theorem}{\bfseries}{\itshape}
\newtheorem*{Thm*}{Theorem}{\bfseries}{\itshape}
\newtheorem{Cor}{Corollary}{\bfseries}{\itshape}
\newtheorem{Prop}[Cor]{Proposition}{\bfseries}{\itshape}
\newtheorem{Lem}[Cor]{Lemma}{\bfseries}{\itshape}
\newtheorem*{Lem*}{Lemma}{\bfseries}{\itshape}
\newtheorem{Fact}[Cor]{Fact}{\bfseries}{\itshape}
{\bfseries}{\itshape}
\newtheorem{Def}[Cor]{Definition}{\bfseries}{\rmfamily}
{\scshape}{\rmfamily}
\newtheorem{Rem}[Cor]{Remark}{\scshape}{\rmfamily}
{\bfseries}{\itshape}

{\scshape}{\rmfamily}

\renewcommand\ge{\geqslant} \renewcommand\le{\leqslant}
\let\tildeaccent=\~ \let\hataccent=\^
\renewcommand\~[1]{\widetilde{#1}}

\def\<{\left<} \def\>{\right>} \def\({\left(} \def\){\right)}

\let\parasymbol=\S \def\secref#1{\parasymbol\ref{#1}}

\let\polishL=l \def\Zoladek.{\.Zol\c adek}

 \def\const{\operatorname{const}}

 \def\etc.{\emph{etc}.}

\def\R{{\mathbb R}}  
\def\N{{\mathbb N}} \def\Q{{\mathbb Q}}

  \def\S{\varSigma}

\def\poly{\operatorname{poly}}

  \def\cF{{\mathcal F}}

\def\rest#1{{\vert_{#1}}}

\begin{document}

\title[Cell decompositions for restricted sub-Pfaffian sets]{Effective
  cylindrical cell decompositions for restricted sub-Pfaffian sets}

\begin{abstract}
  The o-minimal structure generated by the restricted Pfaffian
  functions, known as restricted \emph{sub-Pfaffian} sets, admits a
  natural measure of complexity in terms of a \emph{format} $\cF$,
  recording information like the number of variables and quantifiers
  involved in the definition of the set, and a \emph{degree} $D$
  recording the degrees of the equations involved. Khovanskii and
  later Gabrielov and Vorobjov have established many effective
  estimates for the geometric complexity of sub-Pfaffian sets in terms
  of these parameters. It is often important in applications that
  these estimates are polynomial in $D$.

  Despite much research done in this area, it is still not known
  whether cell decomposition, the foundational operation of o-minimal
  geometry, preserves polynomial dependence on $D$. We slightly modify
  the usual notions of format and degree and prove that with these
  revised notions this does in fact hold. As one consequence we also
  obtain the first polynomial (in $D$) upper bounds for the sum of
  Betti numbers of sets defined using quantified formulas in the
  restricted sub-Pfaffian structure.
\end{abstract}

\author{Gal Binyamini} \address{Weizmann Institute of Science,
  Rehovot, Israel} \email{gal.binyamini@weizmann.ac.il}

\author{Nicolai Vorobjov} \address{Department of Computer Science,
  University of Bath, Bath, BA2 7AY, UK} \email{masnnv@bath.ac.uk}

\thanks{This research was supported by the ISRAEL SCIENCE FOUNDATION
  (grant No. 1167/17) and by funding received from the MINERVA
  Stiftung with the funds from the BMBF of the Federal Republic of
  Germany. This project has received funding from the European
  Research Council (ERC) under the European Union's Horizon 2020
  research and innovation programme (grant agreement No 802107)}


\maketitle

\section{Statement of the main results}
\label{sec:main-statement}

\subsection{Setup}

Let $I:=[0,1] \subset {\mathbb R}$ and for $k, n \in \N,\ k \ge n$,
denote by $\pi^k_n:I^k\to I^n$ the projection map. We sometimes omit
$k$ if its meaning is clear from the context.

Pfaffian functions, introduced by Khovanskii in \cite{khov:dan,
  khov:book}, are analytic functions satisfying triangular systems of
Pfaffian (first order partial differential) equations with polynomial
coefficients.  We refer the reader to
\cite{gv:complexity-computations} for precise definition of Pfaffian
functions, based on {\em Pfaffian chain}, and examples of Pfaffian
functions in an open domain $G \subset {\mathbb R}^k$, which we assume
here for simplicity to be given by a product of intervals.

\begin{Def}[semi-Pfaffian set]
  Let $G $ be an open set in ${\mathbb R}^k$ and $I^k \subset G$.  A
  set $X \subset I^k$ is called \emph{(restricted) semi-Pfaffian} if
  it consists of points in $I^k$ satisfying a Boolean combination of
  atomic equations and inequalities of the kind $f=0$ or $f>0$, where
  $f$ is a Pfaffian functions defined in $G$. The \emph{format} of $X$
  is the number of variables $k$ and the {\em degree} of $X$ is the
  sum of degrees of all the Pfaffian functions appearing in the atomic
  formulas (i.e. the degrees of the polynomials defining these
  Pfaffian functions).
\end{Def}

Note that the degree of $X$ bounds from above the number of all atomic
equations and inequalities.  We assume that a Pfaffian chain has been
fixed once and for all, and all semi-Pfaffian sets under consideration
are defined from this single Pfaffian chain.

\begin{Def}[sub-Pfaffian set]\label{def:sub-pfaff}
  A set $Y \subset I^n$ is called {\em (restricted)
    sub-Pfaffian}\footnote{Below we will consider only
    \emph{restricted} sub-Pfaffian sets, and refer to them simply as
    sub-Pfaffian.} if $Y = \pi^k_n (X)$ for a semi-Pfaffian set
  $X \subset I^k$.
\end{Def}

In the special case of a {\em semi-algebraic set} $X$, the
Tarski-Seidenberg theorem states that the set $Y = \pi^k_n (X)$ is
also semi-algebraic, i.e., is a set of points satisfying a Boolean
combination of polynomial equations and inequalities.  By contrast, a
sub-Pfaffian set may not be semi-Pfaffian. 

It is customary in the literature to define the format and degree of a
sub-Pfaffian set as in Definition~\ref{def:sub-pfaff} to be the format
and degree of the set $X$. Below we introduce a variant of these
notions which turns out to behave better with respect to cell
decompositions. To avoid confusion we refer to these modified notions
as *-format and *-degree. See Remark~\ref{rem:comparing-formats} for a
comparison between the starred and unstarred versions of format and
degree.


\begin{Def}[*-format and *-degree of projections]\label{def:format}
  If $\{X_\alpha\},\ X_\alpha \subset I^{k_\alpha}$ is a finite
  collection of semi-Pfaffian sets and $X_\alpha^\circ$ is a connected
  component of $X_\alpha$, we define the *-format of the sub-Pfaffian
  set
  $Y:=\bigcup_\alpha \pi^{k_\alpha}_n(X_\alpha^\circ) \subset I^n$,
  with this particular representation as a union of projections, to be
  the maximum among the formats of sets $X_\alpha$, and the *-degree
  of $Y$ to be the sum of the degrees of sets $X_\alpha$.
\end{Def}

Strictly speaking, the *-format $\cF$ and *-degree $D$ is associated
to a particular presentation of a sub-Pfaffian set in the form
prescribed by Definition~\ref{def:format}, with different
presentations possibly giving rise to different pairs $(\cF,D)$. By a
slight abuse of notation we will say that a sub-Pfaffian set has
format $\cF$ and degree $D$ if there exists \emph{some} presentation
having this data.

We remark that since the restricted sub-Pfaffian sets form an
o-minimal structure, the connected components of semi-Pfaffian (or
even sub-Pfaffian) sets, as well as their projections, are again
sub-Pfaffian, so the sets $Y$ for which *-format and *-degree are
introduced in Definition~\ref{def:format} are indeed sub-Pfaffian.

\begin{Rem}
  It may be useful below for the reader to consider the notions of
  *-format and *-degree as sub-Pfaffian analogs of the notions of
  dimension and degree in the theory of algebraic or semi-algebraic
  geometry. Our main objective is to obtain, as in the semialgebraic
  case, bounds that depend polynomially on the degree for a fixed
  format.
\end{Rem}

In what follows we write that $a$ is $\const(b)$ (resp., $a$ is
$\poly_b(c)$) as shorthand notation meaning that $a \le \gamma(b)$
(resp., $a \le (c+1)^{\gamma(b)}$) where $\gamma(\cdot)$ is some
\emph{universally fixed} function (which may be different for each
occurrence of this notation in the text). Here $a$ and $c$ denote
natural numbers, and $b$ can involve one or several arguments of any
type.  We also allow several arguments in $\poly_b(c_1,\ldots,c_n)$
which we interpret as $\poly_b(c_1+\cdots+c_n)$. For each occurrence
of an asymptotic notation in this text, the implied functions can be
effectively and explicitly computed from the data defining the
Pfaffian chain (degrees of the differential equations involved), and
we omit these computations in the interest of brevity.

\subsection{Cell decompositions}

We recall the standard definitions of a cylindrical cell and a
cylindrical cell decomposition. Later in the paper we consider only
cylindrical decompositions and omit the prefix ``cylindrical'' for
brevity.

\begin{Def}[Cylindrical cell] A \emph{cylindrical cell} is defined by
  induction as follows.
  \begin{enumerate}
  \item Cylindrical 0-cell in ${\mathbb R}^n$ is an isolated point.
  \item Cylindrical 1-cell in ${\mathbb R}$ is an open interval
    $(a,b) \subset {\mathbb R}$.
  \item For $n \ge 2$ and $0 \le \ell < n$ a cylindrical
    $(\ell+1)$-cell in ${\mathbb R}^n$ is either a graph of a
    continuous bounded function $f:\ C \to {\mathbb R}$, where $C$ is
    a cylindrical a cylindrical $(\ell+1)$-cell in
    ${\mathbb R}^{n-1}$, or else a set of the form
    \begin{multline}
      \{(x_1, \ldots ,x_n) \in {\mathbb R}^n : \ (x_1, \ldots , x_{n-1}) \in C \\
      \text{ and } f(x_1,\ldots,x_{n-1})<x_n<g(x_1,\ldots ,x_{n-1}) \},
    \end{multline}
    where $C$ is a cylindrical $\ell$-cell in ${\mathbb R}^{n-1}$, and
    $f,g:\ C \to {\mathbb R}$ are continuous bounded functions such that
    $f(x_1, \ldots , x_{n-1})<g(x_1, \ldots , x_{n-1})$ for all points
    $(x_1, \ldots ,x_{n-1}) \in C$.
  \end{enumerate}
\end{Def}

This definition implies that any $\ell$-cell is homeomorphic to an
open $\ell$-dimensional ball.

\begin{Def}[Cylindrical cell decomposition] A \emph{cylindrical cell
    decomposition} ${\mathcal D}$ of a subset
  $A \subset {\mathbb R}^n$ is defined by induction as follows.
  \begin{enumerate}
  \item If $n=1$, then ${\mathcal D}$ is a finite family of pair-wise
    disjoint cylindrical cells (i.e., isolated points and intervals)
    whose union is $A$.
  \item If $n \ge 2$, then ${\mathcal D}$ is a finite family of
    pair-wise disjoint cylindrical cells in ${\mathbb R}^n$ whose
    union is $A$ and there is a cylindrical cell decomposition of
    $\pi^n_{n-1} (A)$ such that $\pi^n_{n-1} (C)$ is its cell for each
    $C \in {\mathcal D}$.
  \end{enumerate}
  Let $B \subset A$.  Then ${\mathcal D}$ is \emph{compatible} with
  $B$ if for any $C \in {\mathcal D}$ we have either $C \subset B$ or
  $C \cap B= \emptyset$ (equivalently, some subset
  ${\mathcal D}' \subset {\mathcal D}$ is a cylindrical cell
  decomposition of $B$).
\end{Def}

Our main result is as follows.

\begin{Thm}\label{thm:main}
  Let $\{Y_\alpha\},\ Y_\alpha \subset I^n$ be a collection of $N$
  sub-Pfaffian sets of *-format $\cF$ and *-degree $D$. Then there
  exists a sub-Pfaffian cell decomposition of $I^n$ compatible with
  each $Y_\alpha$ such that the number of cells is $\poly_\cF(N,D)$,
  their *-format is $\const(\cF)$, and their *-degree is
  $\poly_{\cF}(D)$.
\end{Thm}

The proof of Theorem~\ref{thm:main} is given
in~\secref{sec:proof-main}.

\subsection{Sub-pfaffian sets defined by quantified formulas}

In this paper we will consider formulas in the \emph{restricted
  sub-Pfaffian language}, which we define to be the first-order
language consisting of logical connectives, quantifiers, atomic
predicates of the form $({\bf x} \in Y)$ where $Y$ is a restricted
sub-Pfaffian set, and no function symbols. All formulas referred to
below are assumed to be in this language. We introduce the notion of
*-format and *-degree for quantified formulas in this language as
follows.

\begin{Def}[*-format and *-degree of quantified formulas]\label{def:formula-complexity}
  Let $\phi$ be a formula, not necessarily in a prenex
  form. We define the *-degree $D(\phi)$ to be the sum of $D(Y)$ for
  all $Y$ appearing in the atomic predicates of $\phi$. We inductively
  define the *-format $\cF(\phi)$ as follows.
  \begin{itemize}
  \item For $\phi=({\bf x} \in Y)$ we define $\cF(\phi)$ to be the
    *-format of $Y$.
  \item For $\phi=\bigvee_{j=1}^k\phi_j$
    we define
    $\cF(\phi)=\max_j\cF(\phi_j)$.
  \item For $\phi=\bigwedge_{j=1}^k\phi_j$ we define
    $\cF(\phi)=1+\max_j\cF(\phi_j)$.
  \item For $\phi=\exists {\bf x}\ \phi'$ we define
    $\cF(\phi)=\cF(\phi')$.
  \item For $\phi=\lnot\phi',\forall {\bf x}\ \phi'$ we define
    $\cF(\phi)=1+\cF(\phi')$.
  \end{itemize}
\end{Def}

In particular, the format $\cF(\phi)$ is bounded from above by the
maximum over the formats of the atomic predicates of $\phi$, plus the
depth of the parse-tree for $\phi$.

Definition~\ref{def:formula-complexity} is motivated by the following
theorem, which shows that the *-format and *-degree of a set defined
by a sub-Pfaffian formula $\phi$ can be bounded in terms of
$\cF(\phi)$ and $D(\phi)$.

\begin{Thm}\label{thm:formulas}
  Let $\phi$ be a sub-Pfaffian formula as above, with *-format $\cF$
  and *-degree $D$. Then the set defined by $\phi$ has *-format
  $\const(\cF)$ and *-degree $\poly_{\cF}(D)$.
\end{Thm}

The proof of Theorem~\ref{thm:formulas} is given
in~\secref{sec:proof-formulas}.

\subsection{Upper bound on homologies}

Theorem~\ref{thm:main} implies in particular an upper bound
$\poly_\cF(D)$ for the number of connected components of a
sub-Pfaffian set of *-format $\cF$ and *-degree $D$. Below we
generalize this to an upper bound for the homology of a sub-Pfaffian
set in terms of the *-format and *-degree.

\begin{Thm}\label{thm:bound}
  The sum of the Betti numbers of a sub-Pfaffian set of *-format $\cF$
  and *-degree $D$ is bounded by $\poly_\cF(D)$.
\end{Thm}

The proof of Theorem~\ref{thm:bound} is given
in~\secref{sec:proof-bound}. Combining Theorem~\ref{thm:formulas} with
Theorem~\ref{thm:bound} we obtain an upper bound for the homologies of
(restricted) sub-Pfaffian sets defined by quantified formulas, which
is polynomial in the degrees of the Pfaffian functions involved (for
formulas of a fixed format). Such bounds have been obtained by
Khovanskii \cite{khov:book} for Pfaffian varieties; by Zell
\cite{zell:semi-pfaff} for semi-Pfaffian sets; and by Gabrielov,
Vorobjov and Zell \cite{gvz:betti,gv:approximation} for sets defined
using quantifiers under certain topological
restrictions. See~\secref{sec:comparison} for some discussion of these
results. However in spite of the significant work around the
complexity of sub-Pfaffian sets, a polynomial estimate for general
definable sets as provided by Theorem~\ref{thm:bound} does not seem to
have been previously known.

\subsection{Applications and motivation}

Our goal in this paper is to provide a framework that can be used to
effectivize most of the results of o-minimal geometry in the
restricted sub-Pfaffian structure, with polynomial dependence on the
degree. Indeed, many proofs in the theory of o-minimality can be
carried out verbatim using Theorems~\ref{thm:main}
and~\ref{thm:formulas} to obtain such polynomial bounds. As an example
we have the following.

\begin{Cor}
  Let $A\subset I^n$ be a restricted sub-Pfaffian set of *-format
  $\cF$ and *-degree $D$. Then the topological closure and the smooth
  part of $A$ have *-format $\const(\cF)$ and *-degree $\poly_\cF(D)$.

  Let $f:A\to B$ be a restricted sub-Pfaffian map of *-format $\cF$
  and $*$-degree $D$. Then for every $r\in\N$ the $C^r$-smooth locus
  of $f$ has *-format $\const(\cF,r)$ and *-degree $\poly_{\cF,r}(D)$.
\end{Cor}
\begin{proof}
  All of the statements follow by defining the relevant sets using
  first-order formulas in the restricted sub-Pfaffian language and
  applying Theorem~\ref{thm:formulas}.
 
  For instance, to define the smooth part of a set $A$ having
  dimension $m\in\N$ at each point, we can define it as a locus of
  points $x \in A$ such that a neighbourhood of $x$ in $A$ is the
  graph of a $C^1$-smooth map. More precisely, the smooth part of $A$
  is the set of all points $x \in A$ each having a neighbourhood $U$
  in $A$ such that there exists a linear map $T:\R^n \to\R^m$ such
  that the restriction $T|_U$ is one-to-one, $T(U)$ contains a
  neighbourhood of $T(x)$ in $\R^m$, and the inverse to $T|_U$ is a
  $C^1$-smooth map with the Jacobian non-vanishing at every point in
  $T(U)$. It is straightforward to reduce all of this to first-order
  formulas (an ``$\epsilon$-$\delta$ definition''), and we leave the
  details to the reader.
\end{proof}

One of our main motivations for developing the theory in this paper is
in relation to the Pila-Wilkie counting theorem \cite{pila-wilkie} and
its applications in diophantine geometry (see the survey
\cite{scanlon:survey}).

Several applications of the counting theorem involve the geometry of
elliptic curves and abelian varieties --- the most famous example
perhaps being the proof of the Manin-Mumford conjecture by
Pila-Zannier \cite{pz:manin-mumford}. In these applications one
considers sets defined using elliptic and abelian functions. Since
these functions are restricted sub-Pfaffian by an observation of
Macintyre \cite{macintyre:pfaffian}, the definable sets are restricted
sub-Pfaffian. One can therefore hope to effectivize the counting
theorem in this context, and subsequently obtain effective results for
diophantine problems. Toward this end Jones and Thomas
\cite{jt:counting-surfaces} have established a version of the counting
theorem for certain surfaces definable in the restricted sub-Pfaffian
structure, and Jones and Schmidt have explored several applications of
this in diophantine geometry \cite{js:mm-extensions,js:pfaffian-def}.

In an upcoming paper by the first author with Jones, Schmidt and
Thomas, we use the framework developed in the present paper to extend
the result of \cite{jt:counting-surfaces} to arbitrary restricted
sub-Pfaffian sets of arbitrary dimension, and improve the dependence
of the effective constants to make them polynomial in the
degrees. This can be viewed as a case of effectivizing an o-minimal
proof in the sense described above, albeit for a much more technically
involved statement. We expect that this result will greatly extend the
scope of potential applications in diophantine geometry, as well as
give rise to more reasonable (indeed, polynomial in degrees) estimates
in these applications.

\subsection{Comparison with previous results}
\label{sec:comparison}

Gabrielov and Vorobjov have established many results on effectivity of
operations in the restricted sub-Pfaffian category. For a survey we
refer the reader to \cite{gv:complexity-computations}. In these works,
the notion of format and degree for sub-Pfaffian sets is similar to
ours but more straightforward: one considers simply projections of
semi-Pfaffian sets, rather than projections of their connected
components. In particular, in \cite{gv:cell-decomposition} Gabrielov
and Vorobjov prove a cell decomposition result quite similar to our
Theorem~\ref{thm:main}, as follows (where we omit the explicit,
doubly-exponential dependence on the format).

\begin{Fact}\label{fact:main-old}
  Let $\{Y_\alpha \},\ Y_\alpha \subset I^n$ be a collection of $N$
  sub-Pfaffian sets of format $\cF$ and degree $D$. Then there exists
  a linear transformation $L:\R^n\to\R^n$ and a cylindrical cell
  decomposition of $\R^n$, compatible with each $L(Y_\alpha)$, such
  that each cell in the decomposition is sub-Pfaffian and has format
  $\poly_\cF(N,D)$, and the number of cells and their degree is
  $\poly_{\cF}(N, D)$.
\end{Fact}

There are two main differences compared to Theorem~\ref{thm:main}:
first, the application of the linear transformation $L$; and second,
more crucially, the dependence of the format of the cells on the
degree $D$. These are crucial limitation, which make it impossible to
apply this result to obtain analogs of Theorems~\ref{thm:formulas}
and~\ref{thm:bound}: first, because in the recursive proofs it is
essential that the cells preserve the order of coordinates; and
second, more crucially, because after the first recursive application
the format becomes dependent on $D$, and the complexity of any further
operations performed on such cells is no longer polynomial in $D$.

We are also unable to sharpen Fact~\ref{fact:main-old} to eliminate
the dependence of the format on $D$, and this appears to be a
fundamental difficulty\footnote{The linear transformation problem
  seems less fundamental, and could probably be avoided using a
  strategy similar to the one used in the present paper.}. Our main
observation in the present paper is that with the revised notions of
*-format and *-degree it is possible, at a crucial point in the cell
decomposition algorithm, to produce cells with format independent of
$D$. However other aspects of the algorithm become more delicate with
these notions. The main reason is that the known approaches to
effective cell decomposition involve reductions using topological
closure and frontier, and in our setting one must take care to avoid
different components becoming glued along their common boundary when
performing these operations. For this reason the strategy of cell
decomposition employed in the present paper differs significantly from
that of \cite{gv:cell-decomposition}.

\begin{Rem}[Comparing the two notions of format]\label{rem:comparing-formats}
  Since the number of connected components of a semi-Pfaffian set
  $X_\alpha$ of format $\cF$ and degree $D$ is bounded by
  $\poly_\cF(D)$ according to Fact~\ref{fact:semi-pfaff-betti}, it is
  clear that a sub-Pfaffian set of format $\cF$ and degree $D$ has
  *-format $\cF$ and *-degree $\poly_\cF(D)$.

  To obtain bounds in the converse direction, it is necessary to compute
  the format and degree (as a sub-Pfaffian set) of a connected
  component $X_\alpha^\circ$, for a given semi-Pfaffian set $X_\alpha$
  of format $\cF$ and degree $D$. Currently we only know how to do
  this by computing a cell-decomposition using
  Fact~\ref{fact:main-old}, giving $\poly_\cF(D)$ cells of format
  $\poly_\cF(D)$. Accordingly a sub-Pfaffian set of *-format $\cF$ and
  *-degree $D$ has format and degree $\poly_\cF(D)$. The dependence of
  the format on $D$ is what makes the notions of (unstarred) format
  and degree much less efficient. We do not venture to conjecture
  whether this dependence can be removed in general, but in any case
  this appears to be a highly non-trivial problem.
\end{Rem}

We also note that analogs of Theorem~\ref{thm:bound} for quantified
formulas have been pursued, using an entirely different topological
approach, in the work of Gabrielov, Vorobjov and Zell
\cite{gvz:betti}. This paper establishes similar bounds (and also with
an explicit dependence on the format) for sets defined by quantified
formulas in a prenex form:
\begin{equation}
  \{x\in[0,1]^n : Q_1y_1Q_2y_2\cdots Q_\nu y_\nu ((x,y)\in X)\}
\end{equation}
where $X$ is a semi-Pfaffian set that is either open or closed, and
$Q_1,\ldots,Q_\nu\in\{\exists,\forall\}$. For formulas involving only
existential quantifiers, Gabrielov and Vorobjov
\cite{gv:approximation} have managed to remove the topological
condition on $X$ by an approximation method.

\begin{Rem}
  In an unpublished manuscript Clutha \cite{clutha} extended the
  method of \cite{gv:approximation} to an arbitrary number of
  quantifiers, combining with the ideas of \cite{gvz:betti}, and
  claimed in particular a result implying Theorem~\ref{thm:bound} for
  arbitrary formulas. However the proof of this result contains a
  substantial gap and we are presently not able to repair it.
\end{Rem}

\section{Preliminaries on semi-Pfaffian sets}
\label{sec:pfaff-background}

We will require the notion of a (weak) stratification of a
semi-Pfaffian set.

\begin{Def}[\protect{\cite[Definition~5]{gv:stratifications}}]
  A \emph{(weak) stratification} of a semi-Pfaffian set $X$ is a
  subdivision of $X$ into a disjoint union of smooth, not necessarily
  connected, semi-Pfaffian subsets $X_\alpha$ called
  \emph{strata}. The system of equalities and inequalities for each
  stratum $X_\alpha$ of codimension $k$ includes a set of $k$
  equalities $h_{\alpha,1}=\cdots=h_{\alpha,k}=0$ whose differentials
  define the tangent space of $X_\alpha$ at every point of $X_\alpha$.
\end{Def}

We will use the following formulation of the main result of
\cite{gv:stratifications}.

\begin{Fact}[\protect{\cite[Theorem~1]{gv:stratifications}}]\label{fact:stratification}
  Let $X\subset I^n$ be semi-Pfaffian of format $\cF$ and degree
  $D$. Then there is a semi-Pfaffian stratification
  $X=\bigcup_\alpha X_\alpha$ of $X$ where the number of strata and
  their degrees are $\poly_\cF(D)$ and their format is $\const(\cF)$.
\end{Fact}

\begin{Rem}\label{rem:stratification-family}
  We will also need a \emph{parametric version} of
  Fact~\ref{fact:stratification}: if $X\subset I^n\times I^m$ is
  semi-Pfaffian of format $\cF$ and degree $D$, then there exists a
  collection $\{S_\alpha \},\ S_\alpha \subset I^n \times I^m$ of
  semi-Pfaffian sets with their number and degree $\poly_\cF(D)$ and
  their format $\const(\cF)$ such that the following holds. For any
  $x\in I^n$, the collection $\{ S_\alpha\cap\pi_n^{-1}(x) \}$ forms a
  stratification of the fiber $X\cap\pi_n^{-1}(x)$, with
  $\dim (S_\alpha\cap\pi_n^{-1}(x))$ independent of
  $x \in \pi_n (S_\alpha)$ for each $\alpha$.  The proof of this is
  the same as the proof of Fact~\ref{fact:stratification}, treating
  the coordinates in $I^n$ as parameters and performing the
  construction in $I^m$.
\end{Rem}

We also require the following result on the complexities of the {\em
  closure} $\overline X$ and \emph{frontier}
$\partial X:=\overline X\setminus X$ of a semi-Pfaffian set $X$.

\begin{Fact}[\protect{\cite{gabrielov:closure}}]\label{fact:closure}
  Let $X\subset I^n$ be a semi-Pfaffian set of format $\cF$ and degree
  $D$. Then the closure $\overline X$ and the frontier $\partial X$
  are semi-Pfaffian of format $\const(\cF)$ and degree $\poly_\cF(D)$.
\end{Fact}

\begin{Rem}\label{rem:closure-family}
  We will also need a \emph{parametric version} of
  Fact~\ref{fact:closure}: if $X\subset I^m\times I^n$ is
  semi-Pfaffian of format $\cF$ and degree $D$ then the union of the
  fiberwise-closures
  \begin{equation}
    \{(x,y)\in I^{n+m} : y\in \overline X_x \},\quad \text{where}\quad X_x:=\{y:(x,y)\in X\}
  \end{equation}
  is semi-Pfaffian of format $\const(\cF)$ and degree
  $\poly_\cF(D)$. This follows from the proof of
  Fact~\ref{fact:closure}, see
  \cite[Remark~5.4]{gv:complexity-computations}. The same remark holds
  for the fiberwise frontier (by computing the fiberwise closure, and
  then subtracting the $X$).
\end{Rem}

We will also need a bound on the sum of Betti numbers of semi-Pfaffian
sets (in fact we will only require the zeroth Betti number, i.e. the
number of connected components). This type of bound was proved by
Khovanskii \cite{khov:book} for Pfaffian sets, and extended by Zell to
the semi-Pfaffian class. We state only the part we need, omitting the
more precise dependence on the parameters which is achieved in
\cite{zell:semi-pfaff}.

\begin{Fact}[\protect{\cite[Main
    result]{zell:semi-pfaff}}]\label{fact:semi-pfaff-betti}
  The sum of the Betti numbers of a semi-Pfaffian set of format $\cF$
  and degree $D$ does not exceed $\poly_\cF(D)$.
\end{Fact}

\section{Cell decomposition of sub-Pfaffian sets}
\label{sec:semi-Pfaff-cells}

In this section we prove a result on cell decomposition of
sub-Pfaffian sets, that is the main ingredient of the proof of
Theorem~\ref{thm:main}. As a shorthand, if $\Delta$ is a collection of
sets we write $\pi_k(\Delta)$ for the collection of the projections of
the elements of $\Delta$.  Furthermore, $\bigcup \Delta$ will stand
for the union of sets in $\Delta$.  Thus, $\bigcup \pi_k (\Delta)$ is
the union of projections of sets in $\Delta$.  We will freely use the
results on effective bounds for stratifications and frontiers
from~\secref{sec:pfaff-background} without explicit reference.

\begin{Lem}[Effective fiber cutting]\label{lem:fiber-cut}
  Let $X\subset I^{n+m}$ be a semi-Pfaffian set of format $\cF$ and
  degree $D$. Then there exists a semi-Pfaffian set $\widehat X$ of
  format $\const(\cF)$ and degree $\poly_\cF(D)$ such that
  $\pi_n(X)=\pi_n(\widehat X)$ and $\pi_n\rest{\widehat X}$ has
  zero-dimensional fibers.
\end{Lem}

\begin{proof}
  We proceed by induction on the maximal fiber dimension $k$ of
  $\pi_n\rest{X}$. If $k=0$ we are done, and otherwise we will
  construct a set $X'$ with $\pi_n(X)=\pi_n(X')$ and maximal fiber
  dimension smaller then $k$.

  Let $\{S_\alpha \},\ S_\alpha \subset I^{n+m}$ denote a fiberwise
  stratification of $X$, as in Remark~\ref{rem:stratification-family}.
  For any $S_\alpha$ with fiber dimension smaller then $k$, we put
  $S_\alpha$ into $X'$. For any $S_\alpha$ with fiber dimension $k$ we
  do the following.
  \begin{enumerate}
  \item We put the fiberwise frontier of $S_\alpha$ into $X'$ using
    Remark~\ref{rem:closure-family}.
  \item For any $j=1,\ldots,m$ let $C_{\alpha,j}$ be the set of
    fiberwise critical points of the coordinate function $x_{n+j}$ on
    the fibers of $\pi_n\rest{S_\alpha}$. Then we stratify the fibers
    of $\pi_n\rest{C_{\alpha,j}}$ again by
    Remark~\ref{rem:stratification-family} and put all strata of
    fiberwise dimension smaller than $k$ into $X'$.
  \end{enumerate}

  Clearly $\pi_n(X')\subset\pi_n(X)$. To prove the converse, let
  $x\in\pi_n(X)$. If the fiber of $X$ over $x$ has strata of dimension
  less then $k$, or strata of dimension $k$ with non-empty frontier,
  we are done. Otherwise this fiber consists of smooth compact
  $k$-dimensional manifolds. Then one of the coordinate functions
  $x_{n+j}$ must be non-constant on this fiber, and the corresponding
  set $C_{\alpha,j}$ is non-empty and locally closed in the
  fiber. Then the fiberwise stratification of $C_{\alpha,j}$ must
  contain strata of dimension less than $k$ that we put into
  $X'$. This proves the claim.
\end{proof}

The following proposition is quite similar to Theorem~\ref{thm:main}:
the difference is that here we essentially assume that we are given
sub-Pfaffian sets with bounded format and degree (rather than their
*-analogs), but produce cell decompositions with bounded *-format and
*-degree. In~\secref{sec:proof-main} we show how to deduce the general
case of Theorem~\ref{thm:main} from this statement.

\begin{Prop}\label{prop:main}
  Let $\{X_\alpha\},\ X_\alpha \subset I^\ell$, where $\ell \ge n$, be
  a collection of $N$ semi-Pfaffian sets of format $\cF$ and degree
  $D$. Then there exists a sub-Pfaffian cell decomposition of $I^n$
  compatible with each $\pi_n(X_\alpha)$ such that the number of cells
  is $\poly_\cF(N,D)$, their *-format is $\const(\cF)$ and their
  *-degree is $\poly_{\cF}(D)$.
\end{Prop}

\begin{proof}
  We will work by lexicographic induction on $(n,k)$ where
  $$k:=\max_\alpha\dim\pi_{n-1}(X_\alpha).$$
  By Lemma~\ref{lem:fiber-cut} we may assume that
  $\pi_n\rest{X_\alpha}$ has zero-dimensional fibers. Refining each
  $X_\alpha$ into its stratification, we may further assume without
  loss of generality that $X_\alpha$ is smooth, that
  $\pi_n\rest{X_\alpha}$ has rank $\dim\pi_n(X_\alpha)$, and that
  $\pi_{n-1}\rest{X_\alpha}$ has constant rank. Let
  $\Pi :=\{X_\alpha\}$. Using Fact~\ref{fact:closure} we may also
  assume that $\Pi$ is closed under taking frontiers.

  Let $\Pi_k$ (resp. $\Pi_{k+1}$ and $\Pi_{<k}$) denote the collection
  of sets $X_\alpha$ with $\dim X_\alpha=k$ (resp. $\dim X_\alpha=k+1$
  and $\dim X_\alpha<k$).

  For each $X_\alpha \in\Pi_{k+1}$ we also apply
  Lemma~\ref{lem:fiber-cut} to $\pi_{n-1}\rest{X_\alpha}$ and add the
  resulting sets to $\Pi$ (after stratifying as above). With this
  modification we may assume that
  \begin{equation}\label{eq:Pi-decomp}
    \pi_{n-1}(\Pi_{k+1})\subset\pi_{n-1}(\Pi_k\cup\Pi_{<k}).
  \end{equation}

  We construct a set $\Sigma$ of closed semi-Pfaffian sets of bounded
  *-format and *-degree, satisfying
  $\dim\bigcup \pi_{n-1} (\Sigma)<k$. First, we put the closures of
  the sets in $\Pi_{<k}$ into $\Sigma$.

  Next, we wish to ensure that
  $G:=\bigcup \pi_{n-1}(\Pi)\setminus \bigcup \pi_{n-1}(\Sigma)$ is
  smooth, and that for every $X_\alpha \in\Pi_{k+1}$ the intersection
  $\pi_n(X_\alpha)\cap (G\times I)$ is open in
  $G\times I$. By~\eqref{eq:Pi-decomp} and since we already put
  $\Pi_{<k}$ in $\Sigma$ we have
  $G=\bigcup \pi_{n-1}(\Pi_k)\setminus \bigcup \pi_{n-1}(\Sigma)$.
  Since $\pi_{n-1}\rest X_\alpha$ has constant rank $k$ for every
  $X_\alpha\in\Pi_k$, the set $\bigcup \pi_{n-1}(\Pi_k)$ can be
  thought of as an immersed smooth manifold with possible
  self-intersections. It may be non-smooth only in points
  $p\in \bigcup \pi_{n-1}(\Pi_k)$ where:
  \begin{enumerate}
  \item there exist two points $p_\alpha\in X_\alpha\in\Pi_k$ and
    $p_\beta\in X_\beta\in\Pi_k$ such that
    $p=\pi_{n-1}(p_\alpha)=\pi_{n-1}(p_\beta)$ but the projections of
    the germs $(X_\alpha,p_\alpha)$ and $(X_\beta,p_\beta)$ under
    $\pi_{n-1}$ (which are both smooth manifolds) are different;
  \item the point $p$ is also in
    $\bigcup \pi_{n-1}(\{ \partial X_\alpha |\ X_\alpha \in \Pi_k \})$.
  \end{enumerate}
  We show how to put the points corresponding to these two cases in
  $\Sigma$, and it then follows that $G$ is smooth. The second case is
  rather simple: for any $X_\alpha\in\Pi_{<k}\cup\Pi_k$ we add the
  frontier $\partial X_\alpha$ to $\Sigma$. We also note for later
  purposes that, since the sets $X_\alpha$ are bounded, this implies
  \begin{equation}\label{eq:good-boundary}
    \partial\pi_n(X_\alpha)\subset \bigcup \pi_n (\Sigma).
  \end{equation}
  To handle the first case, consider for every pair
  $X_\alpha,X_\beta\subset I^\ell$ in $\Pi_k$ the set
  \begin{equation}\label{eq:product}
    \{ (x,y) \in X_\alpha \times X_\beta:\ x_1=y_1, \ldots ,x_{n-1}=y_{n-1} \}
    \subset I^\ell\times I^\ell
  \end{equation}
  where $x$ denotes the coordinates on the first $I^\ell$ and $y$ on
  the second.  In what follows we will be routinely adjusting the
  dimension $\ell$ so that the set $I^\ell$ contains the new
  semi-Pfaffian sets we create, in this case the set in
  (\ref{eq:product}).  We stratify this set and put (the closure of)
  every stratum of dimension less than $k$ into $\Sigma$. It is clear
  that these strata cover the points described in the first case. Thus
  we established that $G$ is smooth.

  For $X_\alpha\in\Pi_{k+1}$ we know that $\pi_n\rest{X_\alpha}$ has
  constant rank $k+1$ and $\pi_{n-1}\rest{X_\alpha}$ has constant rank
  $k$. Therefore, the projection
  $X_\alpha\cap\pi_{n-1}^{-1}(G)\to G\times I$ is a submersion, and
  indeed $\pi_n(X_\alpha)\cap(G\times I)$ is open in $G\times I$.

  For later purposes we put one more set into $\Sigma$, controlling
  the possible interactions between different elements of $\Pi_k$. For
  any $X_\alpha,X_\beta\in\Pi_k$ we consider the set
  \begin{equation}\label{eq:Z-def}
    Z_{\alpha,\beta}:=
    (\{ (x,y) \in X_\alpha \times X_\beta:\ x_1=y_1, \ldots ,x_{n}=y_{n} \}
    \subset I^\ell\times I^\ell.
  \end{equation}
  We stratify $Z_{\alpha,\beta}$ and add the closures of the strata of
  dimension at most $k-1$, and the frontier of the strata of dimension
  $k$, to $\Sigma$.

  Now apply induction to obtain a cell decomposition of $I^{n-1}$
  compatible with $\pi_{n-1}(\Pi)$ and with $\pi_{n-1}(\Sigma)$. If a
  cell $C$ is disjoint from $\bigcup \pi_{n-1}(\Pi)$ then $C\times I$
  is a cell compatible with every $\pi_n(X_\alpha)$.

  Let $C$ be a cell contained in
  $G= \bigcup \pi_{n-1}(\Pi)\setminus \bigcup \pi_{n-1}(\Sigma)$. We
  will show how to construct cells over $C$ that are compatible with
  every $\pi_n(X_\alpha)$. First consider
  $X_\alpha\in\Pi_{k+1}$. Recall that
  $\pi_n(X_\alpha)\cap (G\times I)$ is open in $G\times I$. It follows
  that its boundary in $G\times I$ agrees with its frontier. If a cell
  in $C\times I$ is compatible with the boundary of $\pi_n(X_\alpha)$
  (or, equivalently in this case, with the frontier) then it is also
  compatible with $\pi_n(X_\alpha)$ by elementary topology. Since we
  assume that $\Pi$ is closed under taking frontiers, it will be
  enough to construct cells over $C$ compatible with $\pi_n(\Pi_k)$.

  We claim that for any $X_\alpha\in\Pi_k$, the set $\pi_n(X_\alpha)$
  is a union of graph cells over $C$. Consider an arbitrary point
  $c_0\in C$. Then the fiber of $\pi_n(X_\alpha)$ over $c_0$ consists
  of finitely many points. By construction, $\pi_n(X_\alpha)$ maps
  submersively to $C$, so as we move $c_0$ these points locally move
  continuously. Moreover the points must remain in $\pi_n(X_\alpha)$
  as we continue globally to $C$, for otherwise one of these points
  would meet $\partial\pi_n(X_\alpha)$
  contradicting~\eqref{eq:good-boundary} since $C\subset G$ is
  disjoint from $\pi_{n-1}(\Sigma)$. Note that for the same reasons
  $X_\alpha$ itself is also a union of graphs of continuous maps over
  $C$. If $s:C\to I$ is such a section of $\pi_n(X_\alpha)$ we denote
  by $\hat s$ the corresponding extension to a section of $X_\alpha$.

  If $s_\alpha,s_\beta:C\to I$ is a pair of such sections obtained
  from $X_\alpha,X_\beta$ then one of
  \begin{align}
    s_\alpha &< s_\beta   &   s_\alpha &=s_\beta   &   s_\alpha&>s_\beta
  \end{align}
  holds uniformly over $C$. Indeed, otherwise there should be a point
  $c_0\in C$ where $s_\alpha(c_0)=s_\beta(c_0)$ and a neighborhood
  where this does not hold identically. This implies that
  $(c_0,c_\alpha(c_0))$ belongs to a stratum of dimension at most
  $(k-1)$, or to the frontier of a $k$-dimensional stratum of
  $Z_{\alpha,\beta}$, and hence $c_0\not\in G$ contradicting
  $C\subset G$.

  By the above, the set of sections obtained from any of the
  $X_\alpha\in\Pi_k$ is
  $$\{s_1<\cdots<s_q\},$$
  where we may suppose for simplicity that $s_1=0$ and $s_q=1$. Also
  note since $\pi_n(X_\alpha)$ is a union of disjoint graphs of
  continuous functions, their number is bounded by the number of
  connected components of $X_\alpha$. By
  Fact~\ref{fact:semi-pfaff-betti} we therefore get at most
  $\poly_\cF(D)$ sections from each $X_\alpha$, and in total
  $q=\poly_F(N,D)$.

  A cell decomposition of $C\times I$, compatible with each
  $X_\alpha$, is given by the cells
  \begin{multline}
    \{x_n=s_1(x_1,\ldots,x_{n-1})\}, \\
    \{s_1(x_1,\ldots,x_{n-1})<x_n<s_2(x_1,\ldots,x_{n-1})\},\\
    \ldots, \{x_n=s_q(x_1,\ldots,x_{n-1})\}.
  \end{multline}
  We show that each of these cells is a sub-Pfaffian set with
  appropriately bounded *-format and *-degree. Specifically we show
  that the graph of each section $s_j$ over $C$ is a sub-Pfaffian set
  with the appropriate *-bounds, and it is then a simple exercise to
  construct each of the cells above (where for the interval cells one
  repeats the construction for $s_j,s_{j+1}$ and works in the direct
  product).

  Recall that $C=\pi^\ell_n(Z^\circ)$ is the projection of the
  connected component $Z^\circ$ of some semi-Pfaffian set
  $Z\subset I^\ell$ (for an appropriate $\ell$). Suppose that $s_j$ is
  a section corresponding to $X_\alpha$. Then the graph of $s_j$ is a
  connected component of the set $\pi_n(X_\alpha)\cap(C\times I)$. In
  fact, since the section $s_j$ extends to a section of $X_\alpha$,
  the graph of $s_j$ the projection $\pi_n(W_j^\circ)$ of a connected
  component $W_j^\circ$ of the set
  \begin{equation}
    W_j := \{ (x,y) \in Z \times X_\alpha : x_1=y_1, \ldots ,x_{n-1}=y_{n-1} \}
  \end{equation}
  given by $Z^\circ$ in the $x$-coordinates and by the graph of
  $y=\hat s_j(x)$. Here we ordered the coordinates in such a way that
  $\pi_n(x,y)=(y_1,\ldots,y_n)$. This proves that the graph of $s_j$
  is indeed sub-Pfaffian with appropriately bounded *-format and
  *-degree. Since each cell is constructed using one or two of these
  graphs, their *-degree is indeed bounded by $\poly_\cF(D)$ (noting
  in particular that there is no dependence on $N$).

  Finally, we must construct cells covering
  $\bigcup \pi_{n-1}(\Sigma)\times I$ and compatible with every
  $\pi_n(X_\alpha)$. Let
  \begin{equation}
    Y_\alpha:=\left(\bigcup \pi_{n-1} (\Sigma )\times I \right)\cap\pi_n(X_\alpha) = \pi_n(\widehat Y_\alpha),
  \end{equation}
  where it is easy to choose $\widehat Y_\alpha$ to be a semi-Pfaffian
  set of bounded complexity. It will suffice to construct a cell
  decomposition of $I^n$ compatible with
  $\bigcup \pi_n(\Sigma)\times I$ and $\{Y_\alpha\}$ and take from it
  only the cells covering $\bigcup \pi_n(\Sigma )\times I$. This can
  now be achieved by induction on $k$, noting that for any $Z$ in this
  collection $\pi_{n-1}(Z)\subset \pi_{n-1}( \Sigma)$ has dimension
  strictly less than $k$.
\end{proof}

\section{Proofs of the main results}
\label{sec:proofs}

\subsection{Proof of Theorem~\ref{thm:main}}
\label{sec:proof-main}

We will require a couple of elementary lemmas, whose proofs are left
an exercises for the reader.

\begin{Lem}
  Suppose that a cell $C$ is compatible with a set $X$. Then it is
  compatible with each connected component of $X$.
\end{Lem}

\begin{Lem}
  Suppose that a cell decomposition of $I^k$ is compatible with a set
  $A$. Then the induced decomposition on $I^n$ is compatible with
  $\pi^k_n(A)$.
\end{Lem}

We are now ready to finish the proof.  Let
$X_{\alpha,\beta}\subset I^k$ denote semi-Pfaffian sets and
$X_{\alpha,\beta}^\circ$ connected components such that
$Y_\alpha=\bigcup_\beta\pi^k_n(X_{\alpha,\beta}^\circ)$. Use
Proposition~\ref{prop:main} to find a sub-Pfaffian cell decomposition
of $I^k$ compatible with $\{X_{\alpha,\beta}\}$ with suitably bounded
*-format and *-degree. Then by the preceding lemmas this cell
decomposition is compatible with $X_{\alpha,\beta}^\circ$, and the
induced decomposition on $I^n$ is therefore compatible with
$Y_\alpha$.

\begin{Rem}
  Note that in the proof of Theorem~\ref{thm:main}, even though we are
  interested in a cell decomposition of $I^n$, we apply
  Proposition~\ref{prop:main} to obtain a cell decomposition of the
  full space $I^k$. This allows us to ensure the compatibility with
  the projections of each connected component $X_{\alpha,\beta}^\circ$
  separately. Applying Proposition~\ref{prop:main} to the projections
  $\pi_n(X_{\alpha,\beta})$ would not suffice. It is therefore crucial
  at this point that Proposition~\ref{prop:main} produces a cell
  decomposition in the original coordinates, without applying a linear
  transformation (or at least preserving the projection
  $\pi^k_n$). For this reason, the cell decomposition algorithm of
  Gabrielov-Vorobjov \cite{gv:cell-decomposition}, which applies such
  a linear transformation to the coordinates, would not suffice for
  our purposes.
\end{Rem}

\subsection{Proof of Theorem~\ref{thm:formulas}}
\label{sec:proof-formulas}

The proof is a routine recursive argument on the structure of $\phi$.
The statements for existential quantifiers and disjunctions hold by
definition. Up to re-writing
\begin{align}
  \forall{\bf x}\phi &\equiv \lnot\exists{\bf x}(\lnot\phi), \\
  \land_{i=1}^k\phi_i &\equiv \lnot\lor_{i=1}^k(\lnot\phi_i),
\end{align}
everything else follows from the statement for negations
$\phi=\lnot\phi'$. This is proved by constructing a cell decomposition
of $I^n$ compatible with the set $X$ defined by $\phi'$, and taking
the union of all cells disjoint from $X$ in this decomposition.

\subsection{Proof of Theorem~\ref{thm:bound}}
\label{sec:proof-bound}

We first establish an effective result on the existence of
triangulations in the restricted sub-Pfaffian class. Recall that a
\emph{finite simplicial complex} in $\R^n$ is a finite collection
$K=\{\bar\sigma_1,\ldots,\bar\sigma_p\}$ of (closed) simplices
$\bar\sigma_i\subset\R^n$ such that the intersection of any pairs
$\bar\sigma_i\cap\bar\sigma_j$, if not empty, is a common face of
$\bar\sigma_i$ and $\bar\sigma_j$; and such that any face of any
$\bar\sigma_i$ also belongs to $K$. We write $|K|$ for the union of
all simplices in $K$.

\begin{Thm}\label{thm:triangulation}
  Let $Y\subset[0,1]^n$ be a closed sub-Pfaffian set, and
  $X_1,\ldots,X_k\subset Y$ be sub-Pfaffian subsets, and suppose all
  these sets have *-format bounded by $\cF$ and *-degree bounded by
  $D$.

  Then there exists a finite simplicial complex $K$ with vertices in
  $\Q^n$ and a definable homeomorphism $\Phi:|K|\to Y$ such that each
  $X_i$ is a union of images by $\Phi$ of open simplices of
  $K$. Moreover $K$ contains $\poly_\cF(D)$ simplices, and $\Phi$ has
  *-format $\const(\cF)$ and *-degree $\poly_\cF(D)$.
\end{Thm}

A proof of the triangulation theorem in the o-minimal setting can be
found in \cite{vdd}, where it follows a similar proof for the
semialgebraic class in \cite{df:homology}. For convenience we refer
the reader to the alternative presentation given in
\cite[Theorem~4.4]{coste:o-minimality}, which also establishes
Theorem~\ref{thm:triangulation} without the effective estimates for
arbitrary o-minimal structures. Deriving the effective estimates from
this proof in the restricted sub-Pfaffian context is a routine
exercise in the application of Theorems \ref{thm:main} and
\ref{thm:formulas}: one only nees to verify that in the proof of
\cite[Theorem~4.4]{coste:o-minimality}, the triangulating map $\Phi$
is indeed defined by a first-order formula $\phi$ with
$\cF(\phi)=\const(\cF)$ and $D(\phi)=\poly_\cF(D)$. As this
verification is entirely straightforward from the presentation of
loc. cit., we leave the details as an exercise for the reader.

To deduce Theorem~\ref{thm:bound} we first apply
Theorem~\ref{thm:triangulation} with $Y=[0,1]^n$ and $X$ the given
sub-Pfaffian set. We obtain a homeomorphism $\Phi:|K|\to [0,1]^n$.  In
particular, the sum of the Betti numbers of $X$ is equal to that of
$\Phi^{-1}(X)$, which is a union of at most $\poly_\cF(D)$
simplices. This set being semialgebraic, the bound on the sum of Betti
numbers now follows, e.g., from \cite[Theorem~1]{gv:dcg}.

\bibliographystyle{plain} \bibliography{nrefs}

\end{document}